\documentclass[psamsfonts,12pt]{amsart}
\usepackage[margin=1in]{geometry} 
\usepackage{amsmath,amssymb,amsthm, amscd}
\usepackage{amssymb,fge}
\usepackage{tikz-cd}

\usepackage{color}

\DeclareMathOperator{\N}{N}

\newenvironment{Gal}{{\rm Gal\,}}{}

\usepackage[margin=1in]{geometry}  
\usepackage{graphicx}              
\usepackage{amsmath}               
\usepackage{amsfonts}              
\usepackage{amsthm}                
\usepackage{verbatim}              
\usepackage{indentfirst}           
\usepackage{underscore}
\usepackage{enumitem}
\usepackage{relsize}
\usepackage{varwidth}
\usepackage{mathtools} 
\usepackage{hyperref}
\hypersetup{
    colorlinks=true,
    linkcolor=blue,
    filecolor=magenta,      
    urlcolor=cyan,
}
\usepackage{amssymb}
\usepackage[all]{xy}
\makeatletter
\renewcommand*\env@matrix[1][*\c@MaxMatrixCols c]{%
  \hskip -\arraycolsep
  \let\@ifnextchar\new@ifnextchar
  \array{#1}}
\makeatother

\newtheorem{thm}{Theorem} 
\newtheorem{lem}[thm]{Lemma}

\newtheorem{conj}[thm]{Conjecture}
\theoremstyle{remark}
\newtheorem{rmk}{Remark}
\theoremstyle{remark}

\newcommand{\Mod}[1]{\ (\textup{mod}\ #1)}     

\newcommand{\QQ}{\mathbb{Q}}

\title{Galois groups of some iterated polynomials over cyclotomic extensions} 
\author[Wade Hindes]{Wade Hindes}

\begin{document}
\begin{abstract} \normalsize Let $\varphi_p(z)=(z-1)^p+2-\zeta_p$, where $\zeta_p\in\bar{\mathbb{Q}}$ is a primitive $p$-th root of unity for some odd prime $p$. Building on previous work, we show that the $n$-th iterate $\varphi_p^n(z)$ has Galois group $[C_p]^n$, an iterated wreath product of cyclic groups, whenever $p$ is not a Wieferich prime.  \vspace{.5cm}  
\end{abstract}
\maketitle
\renewcommand{\thefootnote}{}
The study of Galois groups of iterates of rational functions has gained much interest in recent years; see, for instance, \cite{Hamblen,Jonessurvey,Jones-Manes,Stoll}. However, there are very few explicit examples known when the degree of the map is at least three. In this note, we build upon recent work in \cite{AIM} and prove the following:
\begin{thm}{\label{thm:surj}} Let $p$ be an odd prime, let $\zeta_p$ be a primitive $p$th root of unity, and let \vspace{.1cm}  
\[\varphi_p(z)=(z-1)^p+2-\zeta_p.\vspace{.1cm}  \]   
If $p$ is not a Wieferich prime, i.e. $2^{p-1}\not\equiv1\Mod{p^2}$, then \vspace{.1cm}  
\[\;\;\;\;\;\Gal_{\QQ(\zeta_p)}(\varphi^n)=[C_p]^n\;\; \text{for all $n\geq1$}.\vspace{.1cm}\]
\end{thm} 
\begin{rmk} The only known Wieferich primes are $1093$ and $3511$. Moreover, there are at least $O(\log X)$ non-Wieferich primes less than $X$, assuming the abc-conjecture \cite{Silverman}.     
\end{rmk} 
To see how the Galois group of $\varphi_p^n$, the $n$-th iterate of $\varphi_p$, naturally sits inside $[C_p]^n$, see \cite[Lemma 3.3]{Juul}. The key new fact in this note, which does not appear in \cite{AIM}, is the following norm calculation: 
\begin{lem}{\label{lem:norm}} Let $p$ be an odd prime, let $\zeta_p$ be a primitive $p$th root of unity, and let \vspace{.1cm}  
\[\varphi_p(z)=(z-1)^p+2-\zeta_p. \vspace{.1cm}  \]   
If $\N_{\QQ(\zeta_p)/\QQ}: \QQ(\zeta_p)\rightarrow\QQ$ denotes the absolute norm, then \vspace{.1cm}  
\[\boxed{\N_{\QQ(\zeta_p)/\QQ}(\varphi_p^n(1))\equiv 2^p-1\Mod{p^2}\;\;\text{for all $n\geq1$.}} \vspace{.1cm} \]  
\end{lem}
\begin{proof} We fix some notation. Let $G_p=\Gal(\QQ(\zeta_p)/\QQ)\cong\mathbb{Z}/(p-1)\mathbb{Z}$ and let $\mathfrak{p}$ be the principal ideal generated by $1-\zeta_p$ in $\mathbb{Z}[\zeta_p]$. It is well known that $\mathfrak{p}\cap\mathbb{Z}=p\,\mathbb{Z}$ is the unique prime lying over $p$ and that $p\,\mathbb{Z}[\zeta_p]=\mathfrak{p}^{p-1}$ is totally ramified. These and other standard facts regarding cyclotomic fields can be found in \cite[\S2.9]{Samuel}.    

From here, Lemma \ref{lem:norm} then follows from a more general fact: \vspace{.05cm} 
\[x\equiv 1\Mod{\mathfrak{p}}\;\;\text{implies}\;\;\N_{\QQ(\zeta_p)/\QQ}(\varphi_p(x))\equiv 2^p-1\Mod{p^2}.\vspace{.05cm}\]  
To see this, note that if $x\equiv 1\Mod{\mathfrak{p}}$, then $(x-1)^p$ is in the ideal $p\cdot\mathfrak{p}$, which is Galois invariant. Hence, for any $\sigma\in\Gal(\QQ(\zeta_p)/\QQ)$ there exists $b_\sigma\in\mathbb{Z}[\zeta_p]$ such that \vspace{.1cm}  
\[\sigma((x-1)^p):=y_\sigma=p\cdot(1-\zeta_p)\cdot b_\sigma.\vspace{.1cm}\] 
Writing $a_\sigma=\sigma(2-\zeta_p)$ for $\sigma\in\Gal(\QQ(\zeta_p)/\QQ)$, we see that \vspace{.05cm}  
\[ \N_{\QQ(\zeta_p)/\QQ}(\varphi_p(x))=\prod_{\sigma\in G_p}\big(y_\sigma+a_\sigma\big).\vspace{.05cm} \]
However, expanding out this product of sums, we see that \vspace{.1cm}  
\begin{equation*}
\begin{split} 
\N_{\QQ(\zeta_p)/\QQ}(\varphi_p(x))&\equiv\sum_{\sigma\in G_p}y_\sigma\cdot\Big(\prod_{\tau\neq\sigma}a_\tau\Big)+\prod_{\sigma\in G_p}a_\sigma\\ 
&\equiv \sum_{\sigma\in G_p}y_\sigma\cdot\Big(\prod_{\tau\neq\sigma}a_\tau\Big)+N_{\QQ(\zeta_p)/\QQ}(2-\zeta)\; \Mod{p^2},
\end{split} 
\end{equation*} 
since any term containing more than a single $y_\sigma$, each of which is already divisible by $p$, must vanish (mod $p^2$). On the other hand, substituting the relation $y_\sigma=p\cdot(1-\zeta_p)\cdot b_\sigma$, we see that 
\[\sum_{\sigma\in G_p}y_\sigma\cdot\Big(\prod_{\tau\neq\sigma}a_\tau\Big)=p\cdot\bigg[(1-\zeta_p)\bigg(\sum_{\sigma\in G_p}b_\sigma\cdot\Big(\prod_{\tau\neq\sigma}a_\tau\Big) \bigg)\bigg].\] 
In particular,   
\begin{equation}{\label{eq:inp}}
(1-\zeta_p)\bigg(\sum_{\sigma\in G_p}b_\sigma\cdot\Big(\prod_{\tau\neq\sigma}a_\tau\Big) \bigg)
\end{equation} 
is an element of $\mathfrak{p}$. Conversely, 
\[(1-\zeta_p)\bigg(\sum_{\sigma\in G_p}b_\sigma\cdot\Big(\prod_{\tau\neq\sigma}a_\tau\Big) \bigg)=\frac{1}{p}\sum_{\sigma\in G_p}y_\sigma\cdot\Big(\prod_{\tau\neq\sigma}a_\tau\Big)\]
is Galois invariant. Hence, (\ref{eq:inp}) is in $\mathfrak{p}\cap\mathbb{Z}=p\mathbb{Z}$. Therefore, the sum of the terms containing a single $y_\sigma$ must also vanish (mod $p^2$), and we deduce that 
\[\N_{\QQ(\zeta_p)/\QQ}(\varphi_p(x))\equiv \N_{\QQ(\zeta_p)/\QQ}(2-\zeta_p) \Mod{p^2}.\]  
However, if $\Phi_p(z)=z^{p-1}+z^{p-2}+\dots z+1$ denotes the $p$th cyclotomic polynomial, then 
\[\N_{\QQ(\zeta_p)/\QQ}(2-\zeta_p)=\prod_{\sigma\in G_p}(2-\sigma(\zeta_p))=\Phi_p(2)=2^{p-1}+2^{p-1}+\dots +2+1=2^p-1,\]
which completes the proof of Lemma \ref{lem:norm}.       
\end{proof}
We now use a slight generalization of \cite[Theorem 25]{Hamblen}; for more details, see \cite[\S2]{AIM}.  
\begin{lem}{\label{lem:max}} Let $p$ be an odd prime, let $\zeta_p$ be a primitive $p$th root of unity, and let \vspace{.1cm}  
\[\varphi_p(z)=(z-1)^p+2-\zeta_p. \vspace{.1cm}  \]   
Fix $n\geq1$. If for all $1\leq m\leq n$ there exists a prime ideal $\mathfrak{q}_m\subset \mathbb{Z}[\zeta_p]$ such that \vspace{.05cm} 
\[v_{\mathfrak{q}_m}(\varphi_p^m(1))\not\equiv0\Mod{p},\vspace{.05cm}\] 
then $\Gal_{\QQ(\zeta_p)}(\varphi^n)=[C_p]^n$.   
\end{lem}
\begin{proof}  Note first that $\varphi_p^n(z)$ is Eisenstein at $\mathfrak{p}$: the constant term of $\varphi_p^n(z)$ is $1-\zeta_p$, and all intermediate terms are divisible by $p$ (and hence divisible by $\mathfrak{p}$ also). In particular, $\varphi_p^n(z)$ is irreducible over $\mathbb{Q}(\zeta_p)$ for all $n\geq1$. 

Furthermore, note that the principal ideals in $\mathbb{Z}[\zeta_p]$ generated by $\varphi_p^n(1)$ and $\varphi_p^m(1)$ are coprime for all $n\neq m$: if $\varphi_p^n(1)\in\mathfrak{q} $ and $\varphi_p^m(1)\in\mathfrak{q}$ for some prime ideal $\mathfrak{q}\subset\mathbb{Z}[\zeta_p]$ and some $n>m$, then $\varphi_p^{n-m}(0)\in\mathfrak{q}$:\vspace{.05cm}
\[\varphi_p^{n-m}(0)\equiv\varphi_p^{n-m}(\varphi_p^{m}(1))\equiv\varphi_p^n(1)\equiv0\Mod{\mathfrak{q}}.\vspace{.05cm}\] 
However, $\varphi_p^s(0)=1-\zeta_p$ for all $s\geq1$, so that $\mathfrak{q}=\mathfrak{p}$. On the other hand, it is easy to check that $\varphi_p^t(1)\equiv 1\Mod{\mathfrak{p}}$ for all $t\geq0$, contradicting our assumption that $\varphi_p^n(1)\in\mathfrak{q}$ (or that $\varphi_p^m(1)\in\mathfrak{q}$). 

In particular, Lemma \ref{lem:max} follows immediately from \cite[Theorem 4.3]{Hamblen}, whose statement and proof can be translated verbatim to unicritical polynomials, i.e.,   polynomials of the form $f(z)=(z-\gamma)^d+c$; see \cite[\S2]{AIM}. 
         
\end{proof} 
Finally, we note the following characterization of Wieferich primes: 
\begin{lem}{\label{wief}} Let $p$ be an odd prime. Then $p$ is Wieferich if and only if $2^p-1$ is a $p$-th power $\Mod{p^2}$. 
\end{lem}
\begin{proof} If $p$ is Wieferich, i.e.  $2^{p-1}\equiv1\Mod{p^2}$, then $2^{p}-1\equiv 1\Mod{p^2}$ is a $p$-th power modulo $p^2$. On the other hand, suppose that $2^p-1\equiv x^p\Mod{p^2}$. Reducing modulo $p$, we see that $x\equiv 1\Mod{p}$. Write $x\equiv 1+tp\Mod{p^2}$ for some $t\in\{0,1,\dots,p-1\}$. Then 
\[x^p=1+p^2t+\sum_{i=2}^{p}\binom{p}{i}(tp)^i\equiv1\Mod{p^2}.\] 
Hence, $2^p-1\equiv 1\Mod{p^2}$, from which it follows that  $2^{p-1}\equiv 1\Mod{p^2}$.   
\end{proof}    
We now use these auxiliary facts to prove our main result. 
\begin{proof}[(Proof of Theorem \ref{thm:surj})] Suppose that $\Gal_{\QQ(\zeta_p)}(\varphi^n)\neq[C_p]^n$ for some iterate $n\geq1$. Then Lemma \ref{lem:max} implies that there exists some $m\leq n$, such that:
\begin{equation}{\label{valuation}}
v_\mathfrak{q}(\varphi_p^m(1))\equiv 0\Mod{p},\;\; \text{for all primes $\mathfrak{q}\subset\mathbb{Z}[\zeta_p]$}. 
\end{equation} 
Note first that $\varphi_p^m(1)$ cannot be a unit: 
\[\N_{\QQ(\zeta_p)/\QQ}(\varphi_p^m(1))\equiv 2^p-1\not\equiv\pm{1}\Mod{p^2}\vspace{.1cm}\]
by Lemma \ref{lem:norm}, Lemma \ref{wief}, and the fact that $p$ is not a Wieferich prime. From here, (\ref{valuation}) implies the ideal equation:\vspace{.05cm}  
\[ \varphi_p^m(1)\cdot\mathbb{Z}[\zeta_p]=\prod_{i=1}^s\mathfrak{q}_i^{e_i}=\prod_{i=1}^s\mathfrak{q}_i^{p\cdot r_i}=\bigg(\prod_{i=1}^s\mathfrak{q}_i^{r_i}\bigg)^p=(I)^p;\vspace{.05cm}\]
here the $\mathfrak{q}_i$ are prime ideals and $I=\prod\mathfrak{q}_i^{r_i}$. We now take the ideal norm \cite[\S3.5]{Samuel} of the equation above and obtain \vspace{.1cm}  
\[\big|\N_{\QQ(\zeta_p)/\QQ}(\varphi_p^m(1))\big|=\N\Big(\varphi_p^m(1)\cdot\mathbb{Z}[\zeta_p]\Big)=\N(I)^p.\vspace{.1cm}\] 
In particular, $|\N_{\QQ(\zeta_p)/\QQ}(\varphi_p^m(1))\big|=\pm\N_{\QQ(\zeta_p)/\QQ}(\varphi_p^m(1))$ is a $p$th power. However, note that the $\pm$ is irrelevant: since $p$ is odd,  the $-1$ can be absorbed into the $p$th power (if needed). Therefore, we deduce that \vspace{.05cm}  
\[\N_{\QQ(\zeta_p)/\QQ}(\varphi_p^m(1))=y^p, \;\;\;\text{for some}\; y\in\mathbb{Z}.\vspace{.1cm}\] 
However, reducing this equation (mod $p^2$), we see that $\N_{\QQ(\zeta_p)/\QQ}(\varphi_p^m(1))\equiv y^p\Mod{p^2}$, a contradiction of Lemma \ref{lem:norm}, Lemma \ref{wief}, and the fact that $p$ is not a Wieferich prime.              
\end{proof} 
Although our proof breaks down, it is likely that the Galois groups of iterates of $\varphi_p$ are still as large as possible when $p$ is a Wieferich prime: otherwise, $\N_{\QQ(\zeta_p)/\QQ}(\varphi_p^m(1))$ is a $p$th power (in $\mathbb{Z}$, not just in $\mathbb{Z}/p^2\mathbb{Z}$) for some $m\geq1$, which is doubtful. With this in mind, we make the following conjecture.  \vspace{.1cm}
 
\begin{conj} Let $p$ be an odd prime and let $\varphi_p(z)=(z-1)^p+2-\zeta_p$ be as above. Then $\Gal_{\mathbb{Q}(\zeta_p)}(\varphi_p^n)=[C_p]^n$ for all $n\geq1$;
\end{conj} 

\indent \textbf{Acknowledgements:} This research began at the May 2016 AIM workshop titled ``The Galois theory of orbits in arithmetic dynamics," and I thank AIM and the organizers of this workshop. I also thank Michael Bush and Nicole Looper for their companion work in \cite{AIM}. Finally, I thank the anonymous referee for their helpful comments and for pointing out Lemma \ref{wief}. 
   
\end{document}